%% file: main.tex
\documentclass[a4paper,11pt]{article}
\usepackage[margin=24mm]{geometry}
\input{preamble.tex}
\title{Contour Computation of Linearized Painlev\'e II and IV Solutions
with Monodromy-Based Error Control}
\author{Oleg M. Kiselev\\[3pt]
\small Innopolis University, Innopolis, Russia\\
\small\texttt{o.kiselev@innopolis.ru}}
\date{}
\begin{document}
\maketitle
\begin{abstract}
\input{abstract.tex}
\end{abstract}
\input{body.tex}
\input{declarations.tex}

\input{references.tex}
\end{document}

%% file: preamble.tex
\usepackage[T1]{fontenc}
\usepackage[utf8]{inputenc}
\usepackage[english]{babel}
\usepackage{amsmath,amssymb,amsthm,mathtools}
\usepackage{graphicx,booktabs,array,tikz}
\usetikzlibrary{arrows.meta}
\usepackage{microtype}
\usepackage[hidelinks]{hyperref}
\usepackage{placeins}
\newtheorem{proposition}{Proposition}
\newcommand{\ii}{\mathrm{i}}
\newcommand{\PII}{\mathrm{P}_{\mathrm{II}}}
\newcommand{\PIV}{\mathrm{P}_{\mathrm{IV}}}
\newcommand{\diag}{\operatorname{diag}}
\newcommand{\norm}[1]{\lVert#1\rVert_{\max}}
\newcommand{\dd}{\,\mathrm{d}}
\DeclareMathOperator{\Sym}{Sym}

%% file: abstract.tex
We study the numerical evaluation of contour integral representations for
solutions of the linearized second and fourth Painlev\'e equations.
The construction combines a nonlinear background solution, continuation of
canonical Lax-pair columns, normalization matching, and quadrature over
cycles with decaying branches in distinct sectors. Accuracy is assessed
using a reference fundamental matrix, differential-equation residuals,
and variations of monodromy data computed independently from the spectral
problem. An exact identity describes the propagation of initial-basis
errors. A stagewise a posteriori refinement criterion uses increments in
monodromy variations to allocate numerical accuracy. The contour algorithm
is compared with the Dormand--Prince method on a regular background, in
a region of rapid solution growth, and through hard loss of stability
followed by approximately regular oscillations. Accumulated offsets in
monodromy variations and subsequent drift are measured separately.
The experiments show how drift depends on spectral conditioning,
quadrature accuracy, and error allocation between initial and subsequent
parts of the computation. Five additional initial-data and parameter
cases demonstrate the dependence of comparative performance on the
nonlinear background.

%% file: body.tex
\section{Introduction}
Linearized equations describe the sensitivity of a nonlinear solution to
its initial data and enter the construction of asymptotic corrections.
Across a transition region, a small change in the initial state may
substantially alter the amplitude or phase of subsequent oscillations.
Numerical continuation of such variations therefore requires control of
both the solution and the conserved quantities connecting different
dynamical regimes. Explicit integral representations provide a way to
evaluate solutions at distant points and to examine the transfer of this
information through a transition layer.

Painlev\'e equations arise as integrable models of nonlinear transition
layers and bifurcation phenomena. Their connection with the second
Painlev\'e equation was studied by Haberman~\cite{Haberman77,Haberman79}.
Suleimanov~\cite{Suleimanov92,Suleimanov94} developed the idea of nonlinear
counterparts of the special functions of wave catastrophes.
A hierarchy of asymptotic problems for all six Painlev\'e equations was
discussed in~\cite{KiselevSuleimanov99}. Hard loss of stability for the
second equation was investigated in~\cite{Kiselev01}; related problems
of separatrix crossing and capture into autoresonance were considered
in~\cite{KiselevGlebov03,KiselevGlebov07}.

The isomonodromic approach associates a nonlinear equation with a
compatible pair of linear systems~\cite{ItsNovokshenov86,FIKN06}.
At fixed equation parameters, the nonlinear evolution preserves the
monodromy data. Differentiating these data with respect to initial
conditions gives conserved quantities for the linearized problem.
They provide numerical diagnostics whose values can be recomputed from
the spectral system at each observation point.

Numerical Riemann--Hilbert methods and their behavior in asymptotic regimes
have been studied, in particular, in~\cite{OlverTrogdon12}.
The present work starts from the integral formulas
of~\cite{Kiselev24,Kiselev26} and investigates their numerical implementation
for linearized solutions on a prescribed nonlinear background.
The spectral systems are solved by local Taylor expansions, whose
coefficients also determine the quadratic kernels and their integrals.

The main result is a computational construction and its verification for
two different spectral structures and several background regimes.
Its algorithmic basis consists of the exact initial-basis error identity
\eqref{eq:initial-error}, the monodromy error decomposition
\eqref{eq:budget-identity}, and the successive-refinement criterion in
Section~\ref{sec:budget}. On the regular $\PII$ example, matching the
solution errors gives Stokes-variation drift smaller by factors of
$1.2$--$1.3$ for the contour implementation. For the selected complex
$\PIV$ background, the factors are $1.9$ and $14$ at two accuracy levels;
the direct method is faster and preserves the Wronskian more accurately.
The hard-loss-of-stability experiment for $\PII$ continues a
nonoscillatory approximation into a sequence of approximately regular
oscillations. Additional tolerance experiments show that allocating
accuracy between different parts of the interval substantially affects
subsequent drift even when the overall solution errors are close.
Five further initial-data and parameter cases test the persistence of
these observations. The RK-to-contour drift ratio ranges from $0.84$ to
$1.54$ for regular $\PII$ and equals $15.87$ and $12.69$ for the two
additional $\PIV$ backgrounds.

The computations are supported by symbolic checks of the Lax-pair
compatibility conditions, kernel identities, changes of variables, and
formal coefficients, together with numerical refinement, residual, and
Wronskian tests. Monodromy variations are evaluated through separate
spectral sensitivity calculations. This combination tests the formulas
and their implementation by complementary procedures.

Section~\ref{sec:formulas} specifies the equations, normalizations, and
contour representations. Section~\ref{sec:algorithm} describes the
computational stages. Error control and monodromy diagnostics are given
in Sections~\ref{sec:error} and~\ref{sec:monodromy}, respectively.
Section~\ref{sec:experiments} presents the numerical experiments;
Section~\ref{sec:discussion} discusses their interpretation and scope.
The appendices contain coefficient recurrences and reproduction parameters.

\section{Linearized equations and contour representations}\label{sec:formulas}
\subsection{General construction}
Let a nonlinear equation be written as $X_x=f(x,X)$ with fixed parameters.
Its variation satisfies
\begin{equation}\label{eq:tangent}
 Y_x=M(x)Y,\qquad M(x)=f_X(x,X(x)).
\end{equation}
For the scalar normal form, write $y_{xx}=U(x)y$ and $Y=(y,y_x)^T$.
The coordinates used below give $\operatorname{tr}M=0$.
The fundamental matrix is normalized by $F(x_0)=I$, hence $\det F(x)=1$.

Let $z(\lambda,x)$ be a column solution of the compatible system
\begin{equation*}
 z_\lambda=A(\lambda,x)z,\qquad z_x=B(\lambda,x)z.
\end{equation*}
For the pairs considered here, quadratic expressions $Q,R$ in $z$ satisfy
\begin{equation}\label{eq:identity}
 (\partial_x^2-U)Q=\partial_\lambda R.
\end{equation}
Their explicit forms are given below. A contour cycle means a finite
linear combination of paths $\gamma_j$, with specified columns $z_j$
and constant coefficients $c_j$:
\begin{equation}\label{eq:cycle}
 y_\Gamma(x)=\sum_j c_j\int_{\gamma_j}Q(z_j(\lambda,x))\dd\lambda.
\end{equation}
All paths are oriented from infinity to a common finite vertex of the
cycle. Where branching occurs, the lifts of the paths are specified.

\begin{proposition}\label{prop:cycle}
Suppose the paths are independent of $x$, and the integrals in
\eqref{eq:cycle}, together with those of the first two $x$-derivatives,
converge locally uniformly in $x$. If the total boundary term
$\sum_jc_j[R(z_j)]_{\partial\gamma_j}$ vanishes, then $y_\Gamma$
satisfies $y_{xx}=Uy$.
\end{proposition}
\begin{proof}
Differentiation under the integral sign and~\eqref{eq:identity} give
$(\partial_x^2-U)y_\Gamma=
\sum_jc_j\int_{\gamma_j}\partial_\lambda R\dd\lambda
=\sum_jc_j[R]_{\partial\gamma_j}=0$.
\end{proof}

For our cycles, exponential decay of the canonical columns eliminates
the boundary terms at infinity, and the finite boundary terms cancel
algebraically. The cycles use distinct decay sectors. A contour with both
asymptotic directions in one sector, contractible within the domain of
analyticity of the same integrand, yields zero by Cauchy's theorem.
Distinct sectors alone do not establish linear independence of two
integral solutions; their Wronskian must also be checked.

\subsection{The second Painlev\'e equation}
For $\PII$, use $q=q(x)$, $w=q_x$, and a fixed parameter $\alpha$:
\begin{equation}\label{eq:pii}
 q_{xx}=2q^3+xq+\alpha,\qquad v_{xx}=(6q^2+x)v.
\end{equation}
Introduce the Pauli matrices
\[
 \sigma_1=\begin{pmatrix}0&1\\1&0\end{pmatrix},\quad
 \sigma_2=\begin{pmatrix}0&-\ii\\\ii&0\end{pmatrix},\quad
 \sigma_3=\diag(1,-1).
\]
In the Flaschka--Newell normalization consistent with the sign $+\alpha$,
\begin{equation*}
 \begin{split}
 A_{II}&=-\ii(4\lambda^2+x+2q^2)\sigma_3
 +(4q\lambda-\alpha/\lambda)\sigma_1-2w\sigma_2,\\
 B_{II}&=-\ii\lambda\sigma_3+q\sigma_1.
 \end{split}
\end{equation*}
For $z=(z_1,z_2)^T$, set
\begin{equation*}
 Q_{II}=z_1^2+z_2^2,\quad
 (Q_{II})_x=-2\ii\lambda(z_1^2-z_2^2)+4qz_1z_2,\quad
 R_{II}=-\frac{\ii}{2}(z_1^2-z_2^2).
\end{equation*}
Then~\eqref{eq:identity} holds with $U=6q^2+x$.
The contribution of $-\alpha\sigma_1/\lambda$ cancels in
$\partial_\lambda(z_1^2-z_2^2)$, so the identity also holds for
nonzero $\alpha$.

The formal phase is $\Omega=4\lambda^3/3+x\lambda$.
The column $z_j$ has exponential factor $e^{-(-1)^j\ii\Omega}$ and
decays in the sector
\begin{equation}\label{eq:pii-sectors}
 S_j=\{\lambda:|\arg\lambda-\varphi_j|<\pi/6\},\quad
 \varphi_j=-\pi/6+j\pi/3,\quad j=0,\ldots,5.
\end{equation}
Its leading vector is $e_1=(1,0)^T$ for even $j$ and $e_2=(0,1)^T$
for odd $j$. Along $\lambda=re^{\ii\varphi_j}$, integration proceeds
from infinite $r$ toward a finite vertex.

For $\alpha=0$, the vertex is zero, an ordinary point of the pair.
For $\alpha\ne0$, zero is a regular singular point with local exponents
$\pm\alpha$. The experiment with $\alpha=5\ii$ uses the nonzero vertex
$\lambda_*=1.5$ and fixed lifts to the logarithmic covering.
The two cycles use the groups
\begin{equation*}
 G_A=(0,1,2,3),\qquad G_B=(2,3,4,5).
\end{equation*}
These are weighted sums of integrals of different canonical columns.
Their validity follows directly from Proposition~\ref{prop:cycle}.

\subsection{The fourth Painlev\'e equation}
Let $q=q(x)$ and $h=h(x)$ satisfy
\begin{equation}\label{eq:piv-background}
 \begin{split}
 q_x&=q(2h+q+2x),\\
 h_x&=-h^2-2h(q+x)-2\theta_\infty-4\theta_0^2/q^2.
 \end{split}
\end{equation}
The parameters $\theta_0,\theta_\infty$ are constant.
Setting $p=(h+x+q/2)/2$ gives $q_x=4pq$ and the scalar equation
\begin{equation*}
 q_{xx}=\frac{q_x^2}{2q}+\frac32q^3+4xq^2+2(x^2-\alpha)q+\frac\beta q,
 \qquad \alpha=2\theta_\infty-1,\quad\beta=-8\theta_0^2.
\end{equation*}
On a domain where $q\ne0$, fix a continuous branch of $\sqrt q$.
Substitution of $\delta q=\sqrt q\,y$ into the variational equation gives
\begin{equation*}
 y_{xx}=U_{IV}y,\qquad
 U_{IV}=x^2-\alpha+6xq+\frac{15}{4}q^2+\frac{12\theta_0^2}{q^2}.
\end{equation*}
The absence of a first-derivative term gives a constant Wronskian.
We use this scalar normal form throughout.

The Jimbo--Miwa pair~\cite{JimboMiwa81}, in the required gauge, is
\begin{equation}\label{eq:piv-lax}
 A_{IV}=\begin{pmatrix}a&b\\c&-a\end{pmatrix},\qquad
 B_{IV}=\begin{pmatrix}\lambda+x+q/2&1\\-\chi&-\lambda-x-q/2\end{pmatrix},
\end{equation}
where
\begin{equation*}
 a=\lambda+x+\frac{qh}{2\lambda},\quad
 b=1-\frac q{2\lambda},\quad
 c=-\chi+\frac{qh^2-4\theta_0^2/q}{2\lambda},\quad
 \chi=qh+2\theta_\infty.
\end{equation*}
For a simultaneous column solution $z$, the quadratic expressions
of~\cite{Kiselev26},
\begin{equation}\label{eq:piv-kernel}
 Q_{IV}=\frac{\sqrt q}{2\lambda}\bigl((h-\lambda)z_1^2-z_1z_2\bigr),
 \qquad R_{IV}=-\sqrt q\left(\lambda+2x+\frac{3q}{2}\right)z_1^2
\end{equation}
satisfy~\eqref{eq:identity} with $U=U_{IV}$.
The derivative $(Q_{IV})_x$ is evaluated analytically using
\eqref{eq:piv-background} and~\eqref{eq:piv-lax}.

The phase at infinity and the integration directions are
\begin{equation*}
 \Theta=\frac{\lambda^2}{2}+x\lambda-\theta_\infty\log\lambda,
 \qquad\varphi_n=n\pi/2.
\end{equation*}
For odd $n$, take the first column with exponential factor $e^\Theta$;
for even $n$, take the second with $e^{-\Theta}$. The corresponding
sectors satisfy $|\arg\lambda-\varphi_n|<\pi/4$.
The local exponents at zero are $\pm\theta_0$.

The spectral normalization and the normalization for simultaneous
solutions of both systems must be distinguished. For
$\Psi_k=H_k e^{\Theta\sigma_3}$, $H_k\to I$,
\[
 (\Psi_k)_x=B_{IV}\Psi_k-\Psi_k b_0\sigma_3,\qquad b_0=x+q/2.
\]
If $\mathfrak d(x)=\exp\int_{x_0}^x b_0(\xi)\dd\xi$ and
$\mathsf D=\diag(\mathfrak d,\mathfrak d^{-1})$, then
$\widehat\Psi_k=\Psi_k\mathsf D$ satisfies both systems.
These simultaneous column solutions enter~\eqref{eq:piv-kernel}.

The two computed cycles have indices and vertices
\begin{equation*}
 G_A=(-1,0,1,2),\quad\lambda_A=1.45;\qquad
 G_B=(1,2,3,4),\quad\lambda_B=-1.45.
\end{equation*}
The angles $0$ and $2\pi$ specify different lifts of the same planar
direction. Identifying them without accounting for monodromy changes
the cycle.

\section{Computational construction}\label{sec:algorithm}
\subsection{Reference solution and branches}
The nonlinear background is computed once and shared by the two methods.
In the high-precision experiments, it is continued by local power series
with algebraic coefficient recurrences. A separate continuation of the
same type computes the reference matrix $F_{\rm ref}$ of
\eqref{eq:tangent}, independently of the contour quadrature.

For $\PIV$, $r=1/q$ and $\ell=\log q$ are also continued:
\begin{equation}\label{eq:branches}
 r_x=-r(2h+q+2x),\qquad\ell_x=2h+q+2x,\qquad
 \sqrt q=\exp(\ell/2).
\end{equation}
The initial value of $\ell$ fixes the branch. Continuous continuation
retains it across the cut of the principal square root.
The identities $rq=1$ and $e^\ell=q$ are monitored.

\subsection{Canonical columns and path geometry}
A truncated formal expansion supplies the initial column at
$\lambda=R e^{\ii\varphi_j}$. Its recurrences are given in
Appendix~\ref{app:series}. The exponential factor at the initial point
is factored out of the numerically continued column and restored after
integration; its square is restored for quadratic kernels.

For homogeneous $\PII$, the paths run radially to zero.
For $\alpha=5\ii$, each path follows its ray to
$1.5e^{\ii\varphi_j}$ and then follows successive chords of the circle
of radius $1.5$ to the positive point $1.5$. Each chord subtends at most
$\pi/6$, with the unwrapped angle changing continuously from
$\varphi_j$ to zero. In particular, for $j=5$ it decreases from
$3\pi/2$ to $0$, retaining the specified lift.
For $\PIV$, the inner radius is $1.45$, each chord subtends at most
$\pi/4$, and the terminal angle is $0$ for $G_A$ and $\pi$ for $G_B$.
All chords avoid zero. These rules specify the paths, their lifts,
and their orientations (Figure~\ref{fig:contours}).

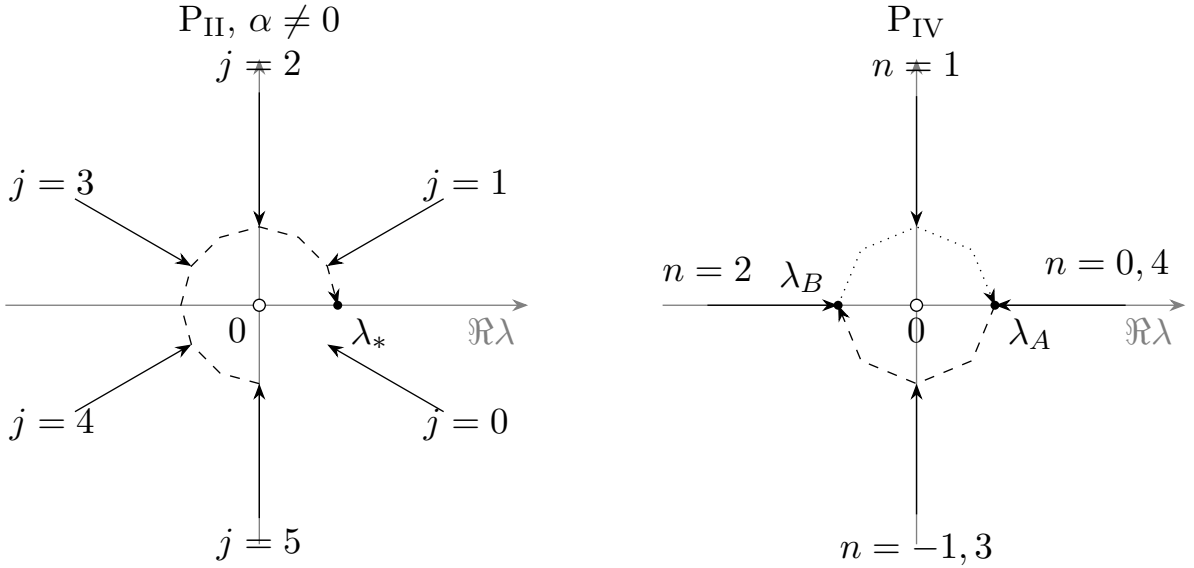
\begin{figure}[htbp]
\centering
\resizebox{.98\linewidth}{!}{%
\begin{tikzpicture}[scale=.78,>=Stealth]
 \begin{scope}
 \draw[->,gray] (-3.4,0)--(3.6,0) node[below left] {$\Re\lambda$};
 \draw[->,gray] (0,-3.2)--(0,3.3);
 \foreach \j in {0,...,5}{
  \pgfmathsetmacro{\ang}{-30+60*\j}
  \draw[->] (\ang:2.85)--(\ang:1.05);
  \node at (\ang:3.2) {$j=\j$};
 }
 \draw[dashed,->] (270:1.05)--(240:1.05)--(210:1.05)--(180:1.05)
 --(150:1.05)--(120:1.05)--(90:1.05)--(60:1.05)--(30:1.05)--(0:1.05);
 \draw[fill=white] (0,0) circle (.08);
 \node[below left] at (0,0) {$0$};
 \fill (1.05,0) circle (.06) node[below right] {$\lambda_*$};
 \node at (0,3.8) {$\PII$, $\alpha\ne0$};
 \end{scope}
 \begin{scope}[xshift=8.8cm]
 \draw[->,gray] (-3.4,0)--(3.6,0) node[below left] {$\Re\lambda$};
 \draw[->,gray] (0,-3.2)--(0,3.3);
 \foreach \ang in {0,90,180,270}{\draw[->] (\ang:2.8)--(\ang:1.05);}
 \node[above] at (0,2.9) {$n=1$};
 \node[above] at (-2.8,.15) {$n=2$};
 \node[below] at (0,-2.9) {$n=-1,3$};
 \node[above] at (2.55,.15) {$n=0,4$};
 \draw[dashed,->] (0:1.05)--(-45:1.05)--(-90:1.05)--(-135:1.05)--(-180:1.05);
 \draw[dotted,->] (180:1.05)--(135:1.05)--(90:1.05)--(45:1.05)--(0:1.05);
 \draw[fill=white] (0,0) circle (.08);
 \node[below] at (0,0) {$0$};
 \fill (1.05,0) circle (.06) node[below right] {$\lambda_A$};
 \fill (-1.05,0) circle (.06) node[above left] {$\lambda_B$};
 \node at (0,3.8) {$\PIV$};
 \end{scope}
\end{tikzpicture}
}
\caption{Projections of the spectral paths. Radial arrows point inward
from infinity. On the left, the dashed line shows the inner part of
branch $j=5$. On the right, the dashed line represents branch $n=4$
of $G_B$ (angle $2\pi\to\pi$), and the dotted line represents branch
$n=2$ of $G_A$ ($\pi\to0$). The other inner parts follow the same
successive-chord rule. An open circle marks the excluded singular point.
The drawing is schematic; numerical radii are specified in the text.
The TikZ implementation was prepared with assistance from OpenAI Codex
(OpenAI).}
\label{fig:contours}
\end{figure}

In a closed subsector of decay, the leading cubic or quadratic phase
gives bounds of the form $C r^m e^{-c r^d}$ for the kernel and its first
$x$-derivatives at sufficiently large $r$, with $d=3$ for $\PII$ and
$d=2$ for $\PIV$. The constants can be chosen locally uniformly in $x$
on compact regular background segments. These bounds ensure convergence
and justify differentiation in Proposition~\ref{prop:cycle}.

\subsection{Matching branches and normalizing the basis}
Write $\Sym^2z=(z_1^2,z_1z_2,z_2^2)^T$.
For each group of four columns, solve at its vertex
\begin{equation}\label{eq:nullspace}
 \sum_{j\in G}c_j\Sym^2z_j(\lambda_G,x_0)=0.
\end{equation}
The coefficients are the alternating minors of the $3\times4$ matrix,
scaled by their largest modulus. Rank three gives a one-dimensional
nullspace. Both rank and residual are checked numerically.
The symmetric squares of simultaneous columns satisfy a common linear
system in $x$. Consequently,~\eqref{eq:nullspace} persists with the same
$c_j$ and cancels the finite boundary term for any quadratic $R$.

Let $\Phi$ have columns $(y_{G_A},(y_{G_A})_x)^T$ and
$(y_{G_B},(y_{G_B})_x)^T$. Provided $\det\Phi(x_0)\ne0$, set
\begin{equation*}
 F_C(x)=\Phi(x)\Phi(x_0)^{-1}.
\end{equation*}
The cycle coefficients and initial matrix are fixed once.
Nonsingularity is verified numerically in the experiments.
Distinct contours may yield dependent periods for exceptional data;
the Wronskian condition is therefore part of the construction.

To match simultaneous-solution normalizations at the vertex, compute
\begin{equation*}
 (T_G)_x=B(\lambda_G,x)T_G,\qquad T_G(x_0)=I.
\end{equation*}
Let $z_j^{\rm loc}$ be a freshly computed spectral column and $z_j^0$
its initial value. Using the component $k$ of largest modulus, set
\begin{equation}\label{eq:gamma}
 \gamma_j=\frac{(T_Gz_j^0)_k}{(z_j^{\rm loc})_k},\qquad
 e_{{\rm dir},j}=
 \frac{\norm{\gamma_jz_j^{\rm loc}-T_Gz_j^0}}{\norm{T_Gz_j^0}}.
\end{equation}
The integral contribution is multiplied by $\gamma_j^2$, while the
remaining component checks the column direction. This transport is
determined by the background and the $B$-system.

\subsection{Quadrature}
At each spectral continuation step, expand $Q$ and $Q_x$ in the local
variable $\zeta$. If $Q=\sum_{n=0}^N k_n\zeta^n$, the contribution of
a step $\Delta\lambda$ is
\begin{equation*}
 \sum_{n=0}^N\frac{k_n}{n+1}(\Delta\lambda)^{n+1}.
\end{equation*}
The derivative of the integral solution is obtained in the same way.
In the $\PII$ experiments, step selection includes the series of the
columns and both accumulated integrals. In the reported $\PIV$
experiments, the local step is selected from the columns, and quadrature
accuracy is assessed by external refinement. This implementation
difference is relevant to interpreting the comparison.

\section{Stagewise error control}\label{sec:error}
\subsection{Nonlinear background and initial asymptotics}
The reference continuation is repeated with higher precision, a smaller
tolerance, and a different series order. The components of the original
nonlinear system, $(q,q_x)$ for $\PII$ and $(q,h)$ for $\PIV$, are
compared first, followed by the reference matrices $F_{\rm ref}$
computed on the corresponding backgrounds. This tests both the
Painlev\'e solution and its variations. For $\PIV$, continuity of the
branch in~\eqref{eq:branches} is also checked.

Canonical-column initialization is tested by varying the radius $R$,
the number $N$ of formal terms, and the continuation accuracy.
The expansion at infinity is asymptotic: increasing $N$ requires checking
that the final terms decrease and the results agree under refinement.
The reported accuracy estimates use these computational comparisons.

\subsection{Spectral continuation and path truncation}
The local step is selected from the last three series coefficients with
safety factor $0.7$; Appendix~\ref{app:series} gives the formula.
Near the regular singular point, it is also limited by a fraction of the
distance to zero. Each contour calculation records the direction error
\eqref{eq:gamma}, the residual in~\eqref{eq:nullspace}, and the boundary
term at the finite truncation radius $R$.

The endpoint value of $R(z_j)$ measures the residual boundary term in
the kernel identity. The omitted integral requires a separate tail
estimate: under $|Q|\le C r^m e^{-c r^d}$, it is bounded by
$C\int_R^\infty r^m e^{-cr^d}\dd r$. The numerical tail checks also
vary $R$.

Matching the simultaneous-solution normalization is sensitive to loss
of column direction, which a scalar factor $\gamma_j$ cannot correct.
The check $\det T_G=1$ is therefore supplemented by the second-component
test and external refinement of the spectral solution.

\subsection{Initial basis, total error, and residual}
Let $\widetilde\Phi=\Phi+\Delta\Phi$ denote the computed contour basis,
$F_C=\Phi\Phi(x_0)^{-1}$, and
$\widetilde F_C=\widetilde\Phi\widetilde\Phi(x_0)^{-1}$.
If both initial matrices are invertible, the exact identity is
\begin{equation}\label{eq:initial-error}
 \widetilde F_C(x)-F_C(x)=
 \bigl(\Delta\Phi(x)-F_C(x)\Delta\Phi(x_0)\bigr)
 \widetilde\Phi(x_0)^{-1}.
\end{equation}
It follows by subtracting
$\widetilde F_C\widetilde\Phi(x_0)=\widetilde\Phi$ and
$F_C\widetilde\Phi(x_0)=\Phi+F_C\Delta\Phi(x_0)$.
An initial-basis error propagates throughout the interval. Reducing the
current quadrature error can therefore leave the total solution error
almost unchanged while substantially reducing its subsequent variation.

On a finite observation set $\mathcal T$, define
\begin{equation}\label{eq:errors}
 E(x)=\frac{\norm{F_{\rm num}(x)-F_{\rm ref}(x)}}{\norm{F_{\rm ref}(x)}},
 \qquad E=\max_{x\in\mathcal T}E(x),\qquad
 W_{\rm err}=\max_{x\in\mathcal T}|\det F_{\rm num}(x)-1|.
\end{equation}
Here and below, $\norm C$ is the largest modulus of a matrix entry.
Normalization is pointwise; division by a global maximum of
$F_{\rm ref}$ can obscure errors following rapid growth.

A further check uses contour solutions at neighboring points to form a
finite-difference approximation of $y_{xx}-Uy$. Refining the difference
grid separates differentiation error from integral-evaluation error.
Residual, Wronskian, and monodromy checks measure complementary
components of the numerical error.

\section{Variations of monodromy data}\label{sec:monodromy}
\subsection{The spectral Jacobian}
Let $m(x,X)$ be the chosen vector of monodromy coordinates.
Along a nonlinear solution, $m_x+m_X f=0$. Differentiation with respect
to initial data gives
\begin{equation*}
 \delta m=J(x)Y(x),\qquad J(x)F(x)=J(x_0).
\end{equation*}
The Jacobian $J$ is expressed in the same tangent coordinates as $Y$.

For $\PII$, write $[a,b]=a_1b_2-a_2b_1$ and compute
\begin{equation*}
 s_j=\frac{[z_j,z_{j+2}]}{[z_j,z_{j+1}]},\qquad j=0,1,2,3,
 \qquad m=(s_0,s_1,s_2,s_3)^T.
\end{equation*}
The indices refer to~\eqref{eq:pii-sectors}.
The four coefficients form a redundant set: there are two independent
complex initial-data variations. Adjacent determinants equal $(-1)^j$
in this normalization; their numerical deviations are recorded.

For $\PIV$, the spectrally normalized multipliers satisfy
$s_{1,x}=-2b_0s_1$, $s_{3,x}=-2b_0s_3$,
$s_{2,x}=2b_0s_2$, and $s_{4,x}=2b_0s_4$.
Constant Stokes matrices are obtained by conjugation,
$\widetilde S_k=\mathsf D^{-1}S_k\mathsf D$.
The computations use the redundant vector of gauge-invariant products
\begin{equation*}
 m=(s_1s_2,s_2s_3,s_3s_4,s_1s_4)^T.
\end{equation*}
These products account for the $x$-dependence of the spectral normalization.

Spectral sensitivity with respect to a component $a$ of the nonlinear
state satisfies the inhomogeneous equation
\begin{equation}\label{eq:sensitivity}
 (z_a)_\lambda=A z_a+A_a z.
\end{equation}
At fixed $x$ and equation parameters, its integral form is
\begin{equation}\label{eq:duhamel}
 z_a(\lambda)=G(\lambda,\lambda_0)z_a(\lambda_0)
 +\int_{\lambda_0}^{\lambda}G(\lambda,\mu)A_a(\mu)z(\mu)\dd\mu,
\end{equation}
where $G$ is the transition matrix of the spectral system along the
specified path. Initial derivatives of the explicit truncated asymptotic
series are evaluated by high-precision central differences.
Equation~\eqref{eq:sensitivity} is then solved, and the determinant
quotient is differentiated. At regular points, the result is checked
against finite differences of spectral solutions and~\eqref{eq:duhamel}.

For $\PIV$, the transformation from $Y=(y,y_x)^T$ to background
variations is
\begin{equation*}
 \begin{pmatrix}\delta q\\\delta h\end{pmatrix}
 =\begin{pmatrix}
 \sqrt q&0\\
 -(h+x+3q/2)/(2\sqrt q)&1/(2\sqrt q)
 \end{pmatrix}Y.
\end{equation*}
The computed Jacobian $m_{(q,h)}$ is multiplied on the right by this
matrix, using the same square-root branch as in the kernel.

At each point, the monitoring Jacobian $J$ is obtained from a fresh
spectral calculation with canonical normalization specified at infinity.
The spectral and sensitivity systems determine $J$; contour quadrature
or direct integration of the linearized system determines $F$.
The procedures share the nonlinear background.

\subsection{Drift normalization and conditioning}
The regular $\PII$ and $\PIV$ experiments use the global indicator
\begin{equation}\label{eq:global-drift}
 D(x)=\frac{\norm{J(x)F_{\rm num}(x)-J(x_0)}}{\norm{J(x_0)}}.
\end{equation}
In the transition experiment, the row scales of $J$ differ by about
twenty orders of magnitude. We therefore use separate denominators
$b_j=\norm{J_j(t_0)}$, where $J_j$ is a Jacobian row, and define
\begingroup\small
\begin{equation}\label{eq:row-drift}
 D_0(t)=\max_j\frac{\norm{C_j(t)-J_j(t_0)}}{b_j},\qquad
 D_{\rm osc}(t)=\max_j\frac{\norm{C_j(t)-C_j(t_r)}}{b_j},\qquad
 C_j=J_jF_{\rm num}.
\end{equation}
\endgroup
All denominators used are nonzero. The point $t_r$ marks entry into the
approximately regular oscillatory regime. The second indicator uses
differences of complex rows. Changing the reference point of the
indicator retains the computed solution and its normalization.

With exact $J$, the monodromy-variation error is
$J(F_{\rm num}-F_{\rm ref})$, together with the reference-solution error.
In the maximum-entry norm,
\begin{equation}\label{eq:conditioning}
 D_j(t)\le\kappa_j(t)E(t)+\eta_j(t),\qquad
 \kappa_j(t)=\frac{2\norm{J_j(t)}\norm{F_{\rm ref}(t)}}{b_j},
\end{equation}
where $\eta_j$ includes spectral-monitoring and reference-matrix errors.
The factor $2$ arises from the two columns of $J$.
The direction of the matrix error also affects drift. Matching the
maxima of $E$ thus leaves the pointwise errors and values of $D$ free
to differ.

\subsection{Accuracy allocation by successive refinement}\label{sec:budget}
Fix a target $\eta>0$ in normalized-drift units, and set
$b=\norm{\widehat J(x_0)}>0$, where $\widehat J$ is the computed spectral
Jacobian. Exact $J,F$ with $JF=J(x_0)$ and approximations
$\widehat J,\widehat F$ satisfy
\begin{equation}\label{eq:budget-identity}
 \widehat J\widehat F-\widehat J(x_0)
 =J(\widehat F-F)+(\widehat J-J)\widehat F
   -(\widehat J(x_0)-J(x_0)).
\end{equation}
Writing $e_F=\norm{\widehat F-F}$ and $e_J=\norm{\widehat J-J}$ gives
\begin{equation}\label{eq:budget-bound}
 \frac{\norm{\widehat J\widehat F-\widehat J(x_0)}}{b}
 \le\frac{2\norm J e_F+2e_J\norm{\widehat F}+e_J(x_0)}{b}.
\end{equation}
The factors $2$ arise from the two summands in each product entry.
Identity~\eqref{eq:initial-error} further separates $e_F$ into current-
and initial-basis contributions.

Computational errors are assessed by successive refinement.
Let $F_C^{(0)},F_C^{(1)},F_C^{(2)}$ be three contour calculations with
decreasing tolerances on a common background; let $F_R^-,F_R^+$ be
reference matrices obtained by refining the background, precision, and
series order; and let $J^-,J^+$ be spectral Jacobians obtained by refining
the radius, formal expansion, and spectral continuation.
All maxima below use the same grid $\mathcal T$, with
$b=\norm{J^+(x_0)}$:
\begin{equation}\label{eq:budget-indicators}
 \begin{aligned}
 c_\ell&=\max_{\mathcal T}
   \frac{\norm{J^+(F_C^{(\ell)}-F_C^{(\ell-1)})}}{b},\quad\ell=1,2,\\
 r&=\max_{\mathcal T}\frac{\norm{J^+(F_R^+-F_R^-)}}{b},\\
 j&=\max_{\mathcal T}
   \frac{\norm{(J^+-J^-)F_C^{(2)}}+\norm{J^+(x_0)-J^-(x_0)}}{b},\\
 d_R&=\max_{\mathcal T}\frac{\norm{J^+F_R^+-J^+(x_0)}}{b}.
 \end{aligned}
\end{equation}
The indicators $r,j,d_R$ monitor the reference calculation and extraction
of monodromy variations. The indicators $c_1,c_2$ measure contour
refinement with spectral sensitivity taken into account.
Accuracy is selected as follows.
\begin{enumerate}
\item Refine the background and spectral Jacobian until
$r,j,d_R\le\eta/5$.
\item Compute three contour levels, including their initial bases.
Check $c_2\le\eta/5$ and $c_2/c_1\le1/2$ for $c_1>0$.
Separately vary the radius and formal-series order, requiring
$\max_{\mathcal T}\norm{J^+\Delta F_C}/b\le\eta/5$.
\item If a threshold is exceeded, refine the corresponding stage and
repeat the check. If the $c_\ell$ stagnate, also refine the radius,
formal-series order, and working precision.
\item For the accepted level, report solution error, Wronskian error,
and drift from the independent spectral monitor.
\end{enumerate}
The fractions $1/5$ and contraction factor $1/2$ specify the safety
margins used here. This is an a posteriori diagnostic based on observed
increments. The rigorous bound~\eqref{eq:budget-bound} requires upper
bounds on $e_F,e_J$;~\eqref{eq:budget-indicators} supplies computational
indicators of these errors. Separately varying $R$ and the series order
tests truncation error shared by the three tolerance levels.

\section{Numerical experiments}\label{sec:experiments}
\subsection{Comparison protocol}
The direct method uses the adaptive Dormand--Prince $5(4)$
pair~\cite{DormandPrince80}. Its rational coefficients and all stages
are evaluated in the same multiple-precision arithmetic as the contour
algorithm. The last step before each observation point is shortened
to end exactly at that point. Absolute tolerance is one tenth of
relative tolerance, with an RMS normalized local-error test.

The direct-method tolerance is selected using $E$ alone.
A pair is considered matched when $1/1.1\le E_{\rm RK}/E_C\le1.1$.
Monodromy-variation drift and Wronskian error are then computed for
the matched pair. All tabulated maxima refer to finite observation grids.

In the regular experiments, timings include preparation of the initial
contour basis, normalization transport, spectral continuation, and
quadrature. The shared background, reference matrix, and independent
Jacobian are excluded from both methods' costs. Tolerance-search time
is recorded separately. The reported timings are single measurements.

\subsection{A regular oscillatory \texorpdfstring{$\PII$}{PII} background}
Set $\alpha=0$, $q(-1)=0.2$, and $q_x(-1)=0.1$, and integrate from
$x=-1$ to $x=-13$, observing $-1,-2,-4,-7,-10,-13$.
Regularity on the negative half-axis follows from an energy argument.
With $s=-x$ and $u(s)=q(-s)$, the energy
\[
 \mathcal E=\frac12u_s^2+\frac{s}{2}u^2-\frac12u^4
\]
satisfies $\mathcal E_s=u^2/2$.
Initially, $\mathcal E(1)=0.0242<1/8$ and $|u(1)|<\sqrt{1/2}$.
Set $D(s)=s^2/8-\mathcal E(s)$. While $u^2<s/2$,
\[
 D_s=\frac{s}{4}-\frac{u^2}{2}>0,\qquad
 D(s)\ge D(1)=0.1008.
\]
If $u(s_1)^2=s_1/2$ at a first point $s_1>1$, continuity gives
$D(s_1)\ge0.1008$. Substitution into the energy also gives
\[
 \mathcal E(s_1)=\frac12u_s(s_1)^2+\frac{s_1^2}{8},\qquad
 D(s_1)=-\frac12u_s(s_1)^2\le0.
\]
The contradiction proves that $u^2<s/2$ persists.
The potential part of the energy is nonnegative in this region, hence
$u_s^2\le2\mathcal E<s^2/4$. Both $(u,u_s)$ remain bounded on each
finite interval $s\ge1$, ensuring continuation over the half-axis.

Both methods use 70 decimal digits; the spectral monitor uses 80.
Table~\ref{tab:pii} gives the results.
\begin{table}[htbp]\centering\small
\caption{Regular $\PII$ background. Indicators are defined in
\eqref{eq:errors} and~\eqref{eq:global-drift}.}\label{tab:pii}
\input{table_pii.tex}
\end{table}
Contour drift is smaller by factors of $1.2$--$1.3$.
The contour calculation is faster at the finest level and slower at
the other two. Spectral continuation and quadrature use order-32 series,
so the timing comparison also reflects the different algorithmic orders.

The 50- and 70-digit reference calculations differ in $F$ by
$9.44\cdot10^{-35}$, and the drift of $JF_{\rm ref}$ is
$1.47\cdot10^{-29}$. Increasing the radius from $5$ to $5.5$, the
number of terms from $64$ to $80$, and refining continuation changes
the Jacobian at $x=-13$ relatively by $1.73\cdot10^{-29}$.
The Wronskian of the unnormalized contour basis at $x=-1$ is
approximately $-2.14031314-1.13613048\ii$.

A five-point finite-difference residual test at $x=-7$ gives
$4.75\cdot10^{-13}$ and $2.97\cdot10^{-14}$ for steps $10^{-3}$ and
$5\cdot10^{-4}$. Their ratio agrees with the fourth-order approximation
of the second derivative. On the zero background, the integral and its
derivative are also checked against the corresponding Airy function.

Controlling quadrature itself is important. At the same internal
tolerance $10^{-12}$, including integral series in step selection
reduces the error from $7.52\cdot10^{-9}$ to $4.03\cdot10^{-18}$;
time changes from $4.24$ to $7.88$ s. This compares two local-error
controllers with the background, radius, and asymptotic order fixed.

\subsection{A complex \texorpdfstring{$\PIV$}{PIV} background with rapid growth}
The independent-variable path and initial data are
\begin{equation}\label{eq:piv-data}
 \begin{gathered}
 x=x_0+t,\quad x_0=0.23+0.07\ii,\quad -6\le t\le6,\\
 q_0=1.17+0.19\ii,\quad p_0=0.41-0.16\ii,\quad
 h_0=2p_0-x_0-q_0/2,\\
 \theta_0=0.37+0.11\ii,\quad\theta_\infty=-0.23+0.07\ii.
 \end{gathered}
\end{equation}
The original binary floating-point representations are embedded exactly
in the calculation;~\eqref{eq:piv-data} displays their decimal notation.
The normalization $F(0)=I$ refers to $(y,y_x)$.
Observation points are $t=-6,-3,0,0.62,3,6$.
Near $t=0.62$, the fundamental matrix exhibits rapid growth and poor
conditioning. The background branch is continued in both directions.
This experiment concerns rapid growth on the specified path; no pole
location or distance-to-pole estimate is assigned to it.

\begin{table}[htbp]\centering\small
\caption{Complex $\PIV$ background, with 70-digit arithmetic for both methods.}
\label{tab:piv}\input{table_piv.tex}
\end{table}
The spectral truncation radius is $14$, with $40$ formal asymptotic
terms. Spectral tolerances are $10^{-23}$ and $10^{-26}$ at the two
levels. RK-to-contour drift ratios are $1.92$ and $14.03$, while
contour-to-RK cost ratios are $9.82$ and $3.50$.
The direct method preserves the Wronskian more accurately
(Table~\ref{tab:piv}).

Errors are strongly nonuniform. At $t=0.62$ on the second level, the
measured contour-matrix difference from the reference is
$4.87\cdot10^{-32}$, compared with $6.95\cdot10^{-16}$ for RK.
The corresponding drifts are $1.58\cdot10^{-26}$ and $2.00\cdot10^{-10}$.
These values quantify local agreement with the reference, whose
accuracy is assessed by separate refinement.
\begin{figure}[htbp]\centering
\includegraphics[width=.98\linewidth]{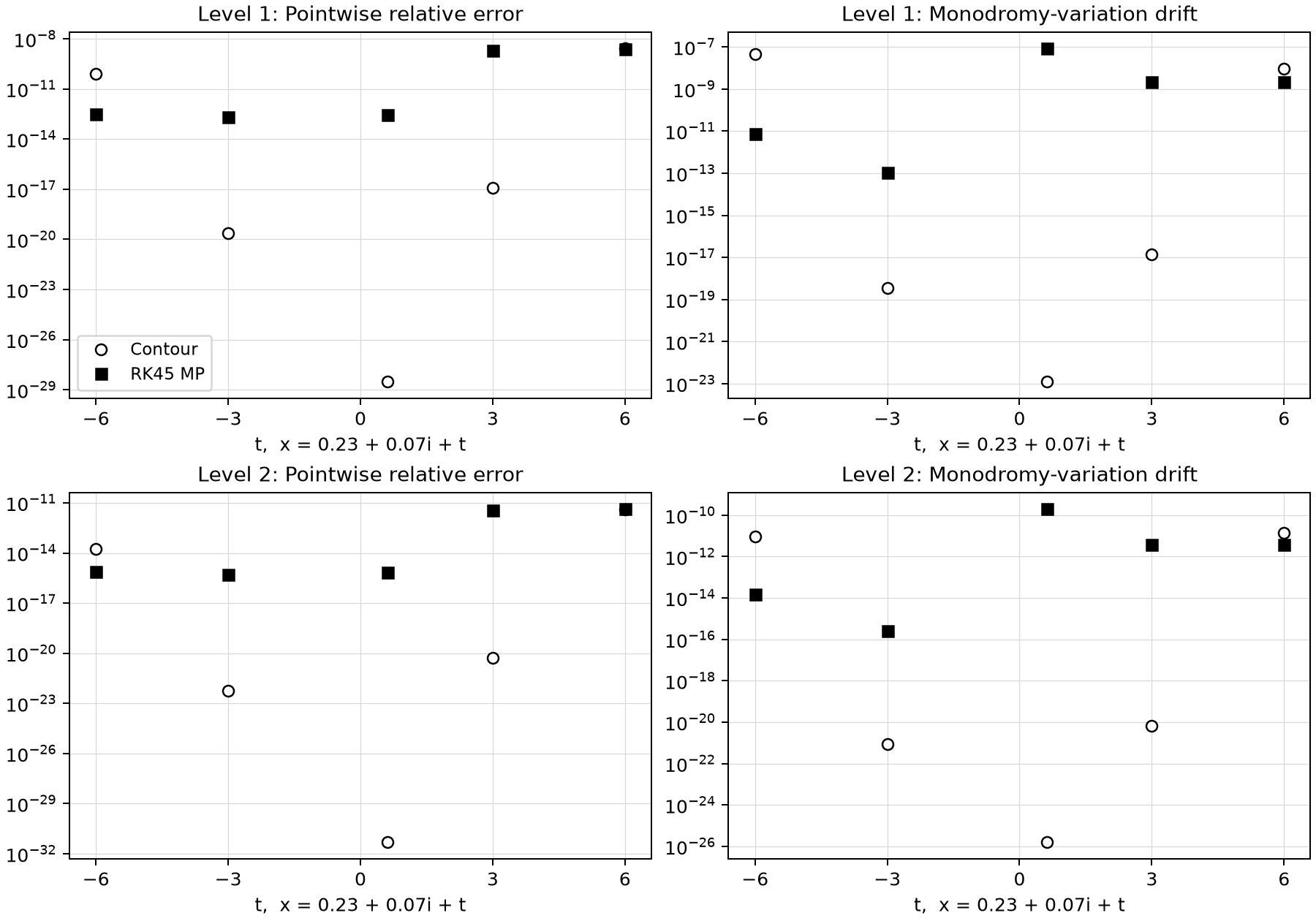}
\caption{Solution errors and monodromy-variation drift at two $\PIV$
accuracy levels. Markers represent computed observation values.
Logarithmic plots show points following initialization.}
\label{fig:piv}
\end{figure}

Refining the spectral monitor to 80 digits, radius $18$, and $56$ formal
terms gives multiplier-product drift $1.45\cdot10^{-31}$ and
$JF_{\rm ref}$ drift at most $1.87\cdot10^{-31}$.
Such independent monitoring is essential over the long path: an early
implementation gave spurious drift of order one or larger even for
the reference matrix.

\subsection{Hard loss of stability for \texorpdfstring{$\PII$}{PII}}
Consider the equation studied in~\cite{Kiselev01},
\begin{equation}\label{eq:rigid}
 \varepsilon^2u_{tt}+2u^3+tu=1,\qquad\varepsilon>0.
\end{equation}
With $p=\varepsilon u_t$, the variational state
$Y=(v,\varepsilon v_t)^T$ satisfies
\begin{equation*}
 Y_t=\varepsilon^{-1}
 \begin{pmatrix}0&1\\-(6u^2+t)&0\end{pmatrix}Y.
\end{equation*}
The transformation to~\eqref{eq:pii} and its variations is
\begin{equation*}
 \begin{gathered}
 x=-t\varepsilon^{-2/3},\quad q=\ii\varepsilon^{-1/3}u,
 \quad w=-\ii\varepsilon^{-2/3}p,\quad\alpha=\ii/\varepsilon,\\
 J_{(u,p)}=J_{(q,w)}\diag(\ii\varepsilon^{-1/3},-\ii\varepsilon^{-2/3}).
 \end{gathered}
\end{equation*}
Variations are taken at fixed $\varepsilon$.
If $\eta$ is evaluated using $Q_{II}$, the physical state is
$(v,\varepsilon v_t)=(-\ii\varepsilon^{1/3}\eta,
\ii\varepsilon^{2/3}\eta_x)$.

The initial point is $t_0=-7$. Let $r(t)$ be the lowest real root of
$2r^3+tr=1$, and put $k=6r^2+t$. Prescribe
\begin{equation*}
 u(t_0)=r(t_0)+\varepsilon^2b(t_0),\quad
 p(t_0)=\varepsilon(r_t(t_0)+\varepsilon^2b_t(t_0)),\quad
 b=\frac{12r^3}{k^4}-\frac{2r}{k^3}.
\end{equation*}
This finite asymptotic approximation specifies an initial-value problem.
Its numerical solution error and the error of approximating the
distinguished asymptotic solution are separate quantities.

The algebraic roots coalesce at $u_*=-2^{-2/3}$,
$t_*=-3\,2^{-1/3}$. In the scaling
$t=t_*+\varepsilon^{4/5}\tau$, $u=u_*+\varepsilon^{2/5}V$, the equation is
\begin{equation*}
 V_{\tau\tau}+6u_*V^2+u_*\tau
 +\varepsilon^{2/5}(2V^3+\tau V)=0.
\end{equation*}
The leading inner profile reduces to the first Painlev\'e equation,
with linearization $Z_{\tau\tau}+12u_*VZ=0$.
This scale explains the need to resolve the transition layer.
A pole of the inner approximation must be distinguished from a
singularity of the exact solution at fixed $\varepsilon$.

On the real axis,~\eqref{eq:rigid} has the coercive energy
$\mathcal H=p^2/2+u^4/2+tu^2/2-u$, with $\mathcal H_t=u^2/2$.
On any bounded interval, a constant $C$ can be chosen such that
$\mathcal H+C\ge c(p^2+u^4+1)$ and
$\mathcal H_t\le C_1(\mathcal H+C)$. These bounds ensure regularity
of the real solution throughout each finite interval. The experiment
therefore follows an asymptotic transition on a regular exact solution.

For $\varepsilon=0.2$, passage through the first large excursion is
computed on $[-7,-1]$, and the same background is then continued to
$t=8$. The initial matrix and contour coefficients are retained.
There are 17 maxima, the first at $t=-1.622750491$.
Entry into the approximately regular regime is defined as the first
maximum at which the relative changes in both period and height from
the preceding values are below $5\%$. This is the ninth maximum,
\begin{equation*}
 t_r=4.044134213\ldots.
\end{equation*}
During the last oscillations, the period changes by $2.3\%$ per cycle
and the peak height by $1.5\%$. Approximately eight periods are followed
after $t_r$, with observations at maxima and intermediate phases.

\begin{table}[htbp]\centering\small
\caption{Continuation of $\PII$ from $t=-7$ to $t=8$ with
$\varepsilon=0.2$. The indicators $D_0,D_{\rm osc}$ are defined
in~\eqref{eq:row-drift}.}
\label{tab:transition}\input{table_transition.tex}
\end{table}
With the constant matched RK tolerance
$1.8829854605461\cdot10^{-14}$, the subsequent-drift ratio is $6.24$
in favor of contour evaluation, whereas the total-drift ratio is $1.46$.
Table~\ref{tab:late} gives pointwise values in the regular regime.
\begin{table}[htbp]\centering\small
\caption{Subsequent drift after $t_r$ for the matched pair.}
\label{tab:late}\input{table_late_drift.tex}
\end{table}

Continuation to larger positive $t$ requires a more accurate spectral
monitor. At $t=8$, the previous tolerance $10^{-32}$ gives an error
of $2.52\cdot10^{-6}$ in the most sensitive row on the reference
solution. At tolerance $10^{-42}$, this decreases to $2.21\cdot10^{-18}$.
Final monitoring of the new points uses 130 digits, $160$ formal terms,
local order $48$, and tolerance $10^{-48}$. Comparison with an
independently refined initial $J(t_0)$ gives residual drift at most
$3.79\cdot10^{-24}$.

At $t=8$, the earlier contour algorithm with tolerance $10^{-14}$ gives
relative solution error $1.91\cdot10^7$ and column-direction error
$0.548$. For the new points in Table~\ref{tab:transition}, spectral
continuation and quadrature are tightened to $10^{-34}$ and local order
$48$, retaining the initial contour basis. Both methods use 90-digit
arithmetic. The experiment identifies the need to revise the settings
used for the first large excursion when continuing to later oscillations.

\subsection{Accuracy allocation after the transition}
In an additional experiment, RK uses tolerance
$5.399977346781\cdot10^{-14}$ up to $t=-1$, then continues the same
approximate state with tolerance $10^{-16}$. The total error remains
$5.50\cdot10^{-13}$, while subsequent drift decreases to
$8.64\cdot10^{-16}$ (last row of Table~\ref{tab:transition}).

For the contour method, the endpoint $t=8$ is separately refined from
tolerance $10^{-34}$ to $10^{-40}$ with the same initial basis.
The deviation of the variations from their previously computed values
at $t_r$ decreases from $2.52\cdot10^{-14}$ to $9.76\cdot10^{-19}$,
while the total solution error remains approximately $5.14\cdot10^{-13}$.
This refinement is an endpoint experiment; the regular-interval values
in the tables retain their original accuracy level.
\begin{figure}[htbp]\centering
\includegraphics[width=\linewidth]{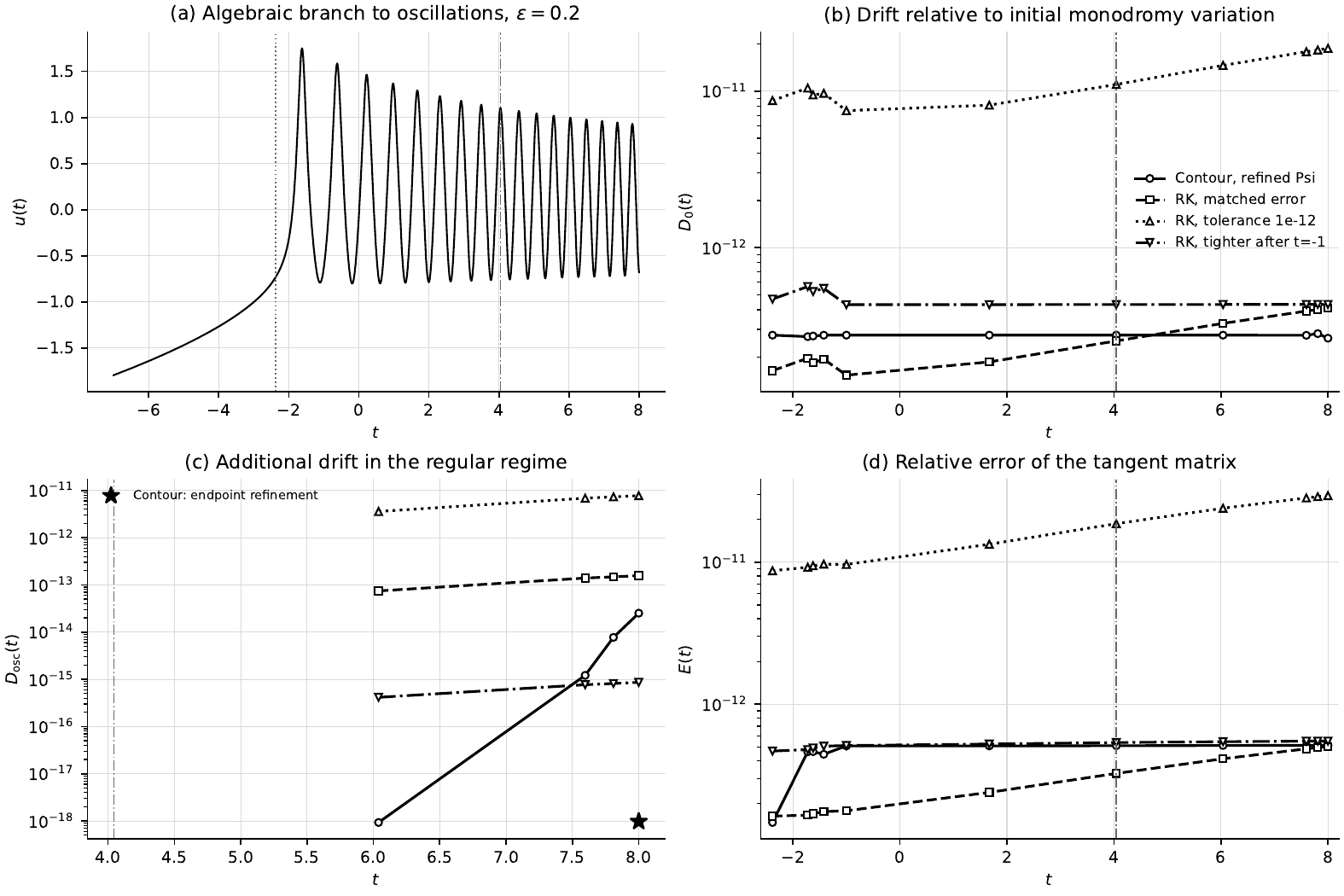}
\caption{Transition from the algebraic approximation to oscillations and
the subsequent drift. The dash-dotted vertical line marks $t_r$.
The star in panel (c) denotes the separately refined endpoint.
Lines connect observation values; initial zeros are omitted from the
logarithmic plots.}\label{fig:rigid}
\end{figure}

In this experiment, refining the background changes the reference
matrices by $1.21\cdot10^{-33}$. Residual drift of the monitoring
$JF_{\rm ref}$ after $t_r$ is $3.79\cdot10^{-24}$, substantially below
the observed method errors. The subsequent-drift indicator separates an
inherited offset from later error accumulation; the initial offset
remains part of the total error.

\subsection{Initial-data and parameter variations}\label{sec:parameters}
Five additional cases are considered. For $\PII$ at $x_0=-1$, choose
$(q_0,w_0)=(0.15,0.1)$, $(0.25,0.1)$, and $(0.2,0.15)$, with the same interval
and observation grid as in the regular experiment. These data satisfy
the energy inequalities ensuring regularity on the negative half-axis.
For $\PIV$, replace $\theta_0$ by $0.35+0.11\ii$ and $0.39+0.11\ii$,
retaining the other data in~\eqref{eq:piv-data}.
All five cases were specified before the comparative calculations.

Each case uses three contour tolerance levels:
$10^{-6},10^{-9},10^{-12}$ for $\PII$ and
$10^{-23},10^{-26},10^{-29}$ for $\PIV$.
The middle level is matched with RK; the third provides refinement.
RK tolerance is selected by the same criterion based on $E$.
Radius and formal-order pairs are $(4,48)$ and $(14,40)$, respectively.
Background accuracy is tested by comparing a 55-digit calculation with
tolerance $10^{-27}$ and order $32$ against a 70-digit calculation
with tolerance $10^{-36}$ and order $40$.

Spectral monitoring uses 80 and 90 digits. For $\PII$, the triples
$(R,N,\tau)$ are refined from $(5,64,10^{-36})$ to
$(5.5,80,10^{-40})$, with local orders $40$ and $44$.
For $\PIV$, the triples are $(18,56,10^{-36})$ and
$(19,64,10^{-40})$, both at local order $32$.
The drifts in Table~\ref{tab:parameters} use the refined Jacobian.
\begin{table}[htbp]\centering\footnotesize
\setlength{\tabcolsep}{3pt}
\caption{Additional cases with matched solution errors. The first three
rows concern $\PII$, and the last two concern $\PIV$.}
\label{tab:parameters}
\input{table_parameters.tex}
\end{table}
\FloatBarrier
For $q_0=0.25$, $w_0=0.1$, RK drift is smaller; in the other two
$\PII$ cases, contour drift is smaller. The two additional $\PIV$
backgrounds retain the contour advantage in this indicator.
Thus both the sign and size of the advantage depend on the background,
even with a common normalization and matched values of $E$.
The global analogue of~\eqref{eq:conditioning} is
$\kappa(x)=2\norm{J(x)}\norm{F_{\rm ref}(x)}/\norm{J(x_0)}$.
Its maximum over the observation grid ranges from $7.27$ to $10.18$
in the additional $\PII$ cases and equals $1.54\cdot10^9$ and
$1.65\cdot10^9$ in the $\PIV$ cases.

The criterion in Section~\ref{sec:budget} is tested with $\eta=10^{-9}$.
A separate contour calculation at the third-level tolerance changes
$(R,N)$ to $(4.5,64)$ for $\PII$ and $(15,48)$ for $\PIV$.
Its difference from the third level is measured by
$\max_{\mathcal T}\norm{J^+\Delta F_C}/b$.
Table~\ref{tab:budget} gives the checks. All five cases satisfy the
stagewise thresholds and contraction condition; independent monitoring
confirms drift below $\eta$.
\begin{table}[htbp]\centering\footnotesize
\setlength{\tabcolsep}{3pt}
\caption{Stagewise refinement in normalized-drift units.}
\label{tab:budget}
\input{table_budget.tex}
\end{table}
\FloatBarrier
Successive contour-increment ratios $c_2/c_1$ range from
$3.13\cdot10^{-4}$ to $3.87\cdot10^{-4}$.
The actual third-level drift is $(2.61\text{--}3.99)\cdot10^{-18}$
for $\PII$, and $1.49\cdot10^{-14}$ and $2.13\cdot10^{-14}$ for $\PIV$.

\FloatBarrier
\section{Discussion}\label{sec:discussion}
The contour representation gives two basis solutions through quadrature
of spectral functions. Its numerical accuracy depends on the complete
sequence: nonlinear background, asymptotic normalization, continuation
of $\Psi$, summation of contour contributions, and inversion of the
initial basis. Sensitivity at any stage can determine the final error.
Column direction and consistent branch continuation are particularly
important.

For regular $\PII$, drift differences are moderate, and the parameter
study contains advantages for both methods. For $\PIV$, larger drift
reductions accompany different pointwise-error distributions and poorer
Wronskian preservation. Solution error, monodromy drift, Wronskian error,
and cost should therefore be reported together. Comparing absolute drift
between equations or monodromy coordinates requires accounting for
normalization and conditioning.

During hard loss of stability, exact monodromy variations are conserved
before, within, and after the transition. Their numerical change in the
regular oscillatory regime is sensitive to the subsequent accuracy of
both methods. Tightening RK tolerance and refining the endpoint contour
quadrature demonstrate the role of error allocation described by
\eqref{eq:initial-error}. Matching the overall solution error $E$ alone
does not determine how accurately invariants are preserved.

The conclusions concern the specified backgrounds, observation grids,
and Dormand--Prince $5(4)$ implementation. In particular, the transition
experiment uses a real solution regular at fixed $\varepsilon$.
Continuation around an actual pole in the complex $x$-plane or through
a change of local coordinates near $q=0$ for $\PIV$ would require
separate tests. Accuracy estimates here are based on refinement and
independent diagnostic identities, with the a posteriori indicators
interpreted as computational evidence rather than interval enclosures.

\section{Conclusion}
A contour algorithm has been implemented for fundamental matrices of
the linearized $\PII$ and $\PIV$ equations, with explicit canonical
columns, paths, and matching conditions. Monodromy variations are checked
independently of contour construction and direct integration.
The experiments cover regular regimes, rapid growth, and hard loss of
stability followed by approximately regular oscillations. Separating
accumulated offsets from subsequent drift identifies the effects of
spectral conditioning and stagewise accuracy allocation on numerical
preservation of monodromy variations.

\appendix
\section{Coefficient recurrences}\label{app:series}
At an ordinary point of a spectral path, let
$z=\sum_{n\ge0}z_n\zeta^n$ and $A=\sum_{n\ge0}A_n\zeta^n$. Then
\begin{equation*}
 z_{n+1}=\frac1{n+1}\sum_{k=0}^n A_kz_{n-k}.
\end{equation*}
For sensitivities, add the coefficient convolution of $A_a z$.
The $1/\lambda$ term is evaluated by series division at
$\lambda_c\ne0$: for $d=z/\lambda$,
$d_n=(z_n-d_{n-1})/\lambda_c$, $d_{-1}=0$.
Quadratic products use $(ab)_n=\sum_{k=0}^n a_kb_{n-k}$.

For a local series of degree $N$, the step satisfies
\[
 |\Delta\lambda|\le0.7\min\left\{h_{\max},\,
 \min_{\substack{a,\ n=N-2,N-1,N\\c_{a,n}\ne0}}
 \left(\frac{\tau\max(1,|c_{a,0}|)}{|c_{a,n}|}\right)^{1/n}\right\}.
\]
Here $c_{a,n}$ are coefficients of the monitored components, $\tau$ is
the internal tolerance, and $h_{\max}$ limits the step and its proximity
to a singularity. The component sets for the two equations are specified
in Section~\ref{sec:algorithm}. The local indicator is supplemented by
external refinement.

The formal matrix $H=I+\sum_{n\ge1}H_n\lambda^{-n}$ satisfies
$H_\lambda=AH-H\Theta_\lambda\sigma_3$.
For $\PIV$, with $A=\lambda\sigma_3+A_0+A_{-1}/\lambda$, its coefficients
satisfy
\begin{equation*}
 [\sigma_3,H_{n+1}]+A_0H_n-xH_n\sigma_3+A_{-1}H_{n-1}
 +\theta_\infty H_{n-1}\sigma_3+(n-1)H_{n-1}=0,
\end{equation*}
for $n\ge0$, $H_0=I$, $H_{-1}=0$.
The off-diagonal entries are determined by the current equation and
the diagonal entries by the next coefficient equation.
For $\PII$, the same substitution uses phase $-\ii\Omega$.
In particular, for $\mathcal H=xq^2+q^4-w^2+2\alpha q$,
\[
 H_1=\frac12\begin{pmatrix}\ii\mathcal H&-\ii q\\
 \ii q&-\ii\mathcal H\end{pmatrix}.
\]
Substitution of recursively generated coefficients into the matrix
identities provides a separate check of the formal normalization.

\section{Parameters and reproducibility}\label{app:reproduce}
Regular $\PII$ uses internal tolerances $10^{-6},10^{-9},10^{-12}$;
all three levels have truncation radius $4$ and formal order $48$.
The two principal $\PIV$ levels are specified after Table~\ref{tab:piv}.
In the transition experiment, the original six points are
$-7,t_*,t_{\rm peak}-0.1,t_{\rm peak},t_{\rm peak}+0.2,-1$,
where $t_{\rm peak}$ is the time of the first maximum. Additional points are
\[
 1.670299696,\quad4.044134213,\quad6.038698755,\quad
 7.596161693,\quad7.809343175,\quad8.
\]
Full saved values of the nodes are used. Identical binary offsets are
embedded exactly in multiple-precision arithmetic for both methods
and the spectral monitor.

Computations use mpmath with gmpy2 acceleration; symbolic identity
checks use SymPy. Tests cover Lax-pair compatibility, quadratic kernels,
changes of variables, formal coefficients, spectral sensitivities,
and properties of the direct integrator. The full suite contains
42 automated tests, complemented by the refinement experiments reported
above.

Tables are generated directly from saved numerical records.
The records retain initial data, arithmetic parameters, monitoring
matrices, trial tolerances, unsuccessful preliminary calculations,
and source-file checksums. For $\PII$, the shared background also
contains $\int q\dd x$, used for normalization transport; its cost is
included in background preparation. Timings for the extended transition
experiment combine reused verified points and an endpoint preflight
calculation, so the timing comparison is confined to the regular
experiments.

The computational supplement contains algorithm implementations,
automated checks, original matrices in high-precision decimal notation,
all refinement parameters, and table-generation tools. A SHA-256
manifest identifies its files. The parameter study is reproduced by
recomputing the backgrounds, spectral sensitivities, contour integrals,
and matched RK trajectories; radius and order refinement is a separate
calculation.

OpenAI Codex (OpenAI) assisted with implementation, examination of
derivations, symbolic and numerical test development, and preparation
of figure-generating code under the author's direction. Reported values
are outputs of the supplied numerical programs. The data plots are
generated from saved numerical records, and the contour diagram is
specified by the paths described in Section~\ref{sec:algorithm}.

%% file: table_pii.tex
\begin{tabular}{clrrrr}
\toprule
Level & Method & $E$ & $D$ & $W_{\rm err}$ & Time (s)\\
\midrule
1 & Contour & $2.236\cdot10^{-9}$ & $1.901\cdot10^{-9}$ & $4.528\cdot10^{-9}$ & 5.18\\
1 & RK & $2.456\cdot10^{-9}$ & $2.475\cdot10^{-9}$ & $4.727\cdot10^{-9}$ & 0.28\\
2 & Contour & $8.728\cdot10^{-13}$ & $6.681\cdot10^{-13}$ & $2.150\cdot10^{-13}$ & 6.41\\
2 & RK & $8.094\cdot10^{-13}$ & $8.075\cdot10^{-13}$ & $1.603\cdot10^{-12}$ & 1.36\\
3 & Contour & $4.031\cdot10^{-18}$ & $3.279\cdot10^{-18}$ & $2.811\cdot10^{-18}$ & 7.88\\
3 & RK & $4.032\cdot10^{-18}$ & $4.012\cdot10^{-18}$ & $8.037\cdot10^{-18}$ & 15.70\\
\bottomrule
\end{tabular}

%% file: table_piv.tex
\begin{tabular}{clrrrr}
\toprule
Level & Method & $E$ & $D$ & $W_{\rm err}$ & Time (s)\\
\midrule
1 & Contour & $2.655\cdot10^{-9}$ & $4.399\cdot10^{-8}$ & $7.159\cdot10^{-8}$ & 44.84\\
1 & RK & $2.422\cdot10^{-9}$ & $8.445\cdot10^{-8}$ & $2.763\cdot10^{-12}$ & 4.56\\
2 & Contour & $4.212\cdot10^{-12}$ & $1.428\cdot10^{-11}$ & $1.549\cdot10^{-11}$ & 53.67\\
2 & RK & $4.309\cdot10^{-12}$ & $2.003\cdot10^{-10}$ & $6.514\cdot10^{-15}$ & 15.32\\
\bottomrule
\end{tabular}

%% file: table_transition.tex
\begin{tabular}{lrrr}
\toprule
Method & $E$ & $\max D_0$ & $\max D_{\rm osc}$\\
\midrule
Contour & $5.139\cdot10^{-13}$ & $2.816\cdot10^{-13}$ & $2.520\cdot10^{-14}$\\
RK, $\tau=10^{-12}$ & $2.943\cdot10^{-11}$ & $1.881\cdot10^{-11}$ & $7.777\cdot10^{-12}$\\
RK, matched & $5.043\cdot10^{-13}$ & $4.099\cdot10^{-13}$ & $1.571\cdot10^{-13}$\\
RK, variable tolerance & $5.500\cdot10^{-13}$ & $5.629\cdot10^{-13}$ & $8.641\cdot10^{-16}$\\
\bottomrule
\end{tabular}

%% file: table_late_drift.tex
\begin{tabular}{rrr}
\toprule
$t$ & $D_{\rm osc}$, contour & $D_{\rm osc}$, RK\\
\midrule
6.038699 & $9.302\cdot10^{-19}$ & $7.461\cdot10^{-14}$\\
7.596162 & $1.215\cdot10^{-15}$ & $1.396\cdot10^{-13}$\\
7.809343 & $7.745\cdot10^{-15}$ & $1.484\cdot10^{-13}$\\
8.000000 & $2.520\cdot10^{-14}$ & $1.571\cdot10^{-13}$\\
\bottomrule
\end{tabular}

%% file: table_parameters.tex
\begin{tabular}{lrrrrr}
\toprule
Data & $E_C$ & $E_{\rm RK}$ & $D_C$ & $D_{\rm RK}$ & $D_{\rm RK}/D_C$\\
\midrule
$q_0=.15,\ w_0=.1$ & $8.846\cdot10^{-13}$ & $8.879\cdot10^{-13}$ & $5.739\cdot10^{-13}$ & $8.832\cdot10^{-13}$ & 1.539\\
$q_0=.25,\ w_0=.1$ & $8.652\cdot10^{-13}$ & $8.679\cdot10^{-13}$ & $1.033\cdot10^{-12}$ & $8.693\cdot10^{-13}$ & 0.841\\
$q_0=.2,\ w_0=.15$ & $8.966\cdot10^{-13}$ & $8.993\cdot10^{-13}$ & $7.486\cdot10^{-13}$ & $8.968\cdot10^{-13}$ & 1.198\\
$\theta_0=.35+.11\ii$ & $4.776\cdot10^{-12}$ & $4.895\cdot10^{-12}$ & $1.290\cdot10^{-11}$ & $2.048\cdot10^{-10}$ & 15.871\\
$\theta_0=.39+.11\ii$ & $3.676\cdot10^{-12}$ & $3.677\cdot10^{-12}$ & $1.574\cdot10^{-11}$ & $1.998\cdot10^{-10}$ & 12.692\\
\bottomrule
\end{tabular}

%% file: table_budget.tex
\begin{tabular}{lrrrr}
\toprule
Data & $c_2$ & $\max(r,j,d_R)$ & Refined $R,N$ & $D_C^{(2)}$\\
\midrule
$q_0=.15,\ w_0=.1$ & $5.738\cdot10^{-13}$ & $9.014\cdot10^{-30}$ & $2.978\cdot10^{-18}$ & $2.614\cdot10^{-18}$\\
$q_0=.25,\ w_0=.1$ & $1.033\cdot10^{-12}$ & $2.965\cdot10^{-29}$ & $3.657\cdot10^{-18}$ & $3.990\cdot10^{-18}$\\
$q_0=.2,\ w_0=.15$ & $7.486\cdot10^{-13}$ & $2.780\cdot10^{-29}$ & $3.430\cdot10^{-18}$ & $3.290\cdot10^{-18}$\\
$\theta_0=.35+.11\ii$ & $1.290\cdot10^{-11}$ & $1.732\cdot10^{-28}$ & $7.904\cdot10^{-24}$ & $1.493\cdot10^{-14}$\\
$\theta_0=.39+.11\ii$ & $1.573\cdot10^{-11}$ & $1.645\cdot10^{-28}$ & $8.647\cdot10^{-24}$ & $2.128\cdot10^{-14}$\\
\bottomrule
\end{tabular}

%% file: declarations.tex
\section*{Funding}
This research received no specific grant from funding agencies in the
public, commercial, or not-for-profit sectors.

\section*{Declaration of competing interest}
The author declares no competing interests.

\section*{Data availability}
The computational supplement accompanying this manuscript contains
the numerical code, automated tests, high-precision output records,
and instructions for reproducing the tables and figures.

\section*{Declaration of generative AI and AI-assisted technologies in the manuscript preparation process}
OpenAI Codex assisted with drafting and translation, examination of
mathematical derivations, and development of symbolic, numerical,
and figure-generation code. The author directed the work and reviewed
and revised the resulting text and arguments. Computational verification
used SymPy, direct contour quadrature, spectral sensitivity calculations,
differential-equation residuals, Wronskians, and successive refinement.
The author takes responsibility for the content of the manuscript.